\documentclass[12pt]{elsarticle}
\usepackage[left=1in,right=1in, top=1.2in,bottom=1.2in]{geometry}
\usepackage{EJGTAart}
\usepackage{times}
\usepackage{amssymb,amsthm,latexsym,amsmath,epsfig,pgf}

\usepackage{mathabx}
\usepackage{algorithm}
\usepackage[noend]{algorithmic}
\usepackage{epstopdf}
\usepackage[normalem]{ulem}

\usepackage{graphicx}
\usepackage{tkz-graph}
\usetikzlibrary{calc,arrows.meta,positioning,fit}
\usetikzlibrary{decorations.markings}

\GraphInit[vstyle = Normal] 
\tikzset{
  LabelStyle/.style = {font = \tiny\bfseries },
  VertexStyle/.append style = { inner sep=5pt,
                                font = \tiny\bfseries},
  EdgeStyle/.append style = {->} }
\usetikzlibrary{arrows, shapes, positioning}

\newcommand{\overundercup}[2]{{{\underset{#1}{\overset{#2}{\fontsize{5pt}{6pt}\selectfont\bigcup}}}}}

\volume{{\bf -} (-)}

\firstpage{1}

\runauth{
A Cartesian graph-decomposition theorem based on a vertex-removing synchronised graph product
\hspace{2ex} $\arrowvert$\hspace{2ex}
Antoon H. Boode}

\newtheorem{theorem}{Theorem}[section]
\newtheorem*{theorem A}{Theorem A}
\newtheorem*{theorem B}{N\"olker's Theorem}

\theoremstyle{remark}

\theoremstyle{remark}

\newtheorem{claim}{Claim}

\usepackage{csquotes}  
\usepackage{crop}

\makeatletter

\theoremstyle{definition}

\begin{document}
\begin{frontmatter}
\papertitle{A Cartesian graph-decomposition theorem based on a vertex-removing synchronised graph product}



\author[label2]{Antoon H. Boode}

\address[label2]{\small {Robotics Research Group, \\InHolland University of Applied Science\\ The Netherlands}

\vspace*{2.5ex} 
 \normalfont ton.boode@inholland.nl}

\begin{abstract}
\noindent
Recently, we have introduced and modified two graph-decomposition theorems based on a new graph product, motivated by applications in the context of synchronising periodic real-time processes. 
This vertex-removing synchronised product (VRSP), is based on modifications of the well-known Cartesian product and is closely related to the synchronised product due to W\"ohrle and Thomas. 
Here, we introduce a new graph-decomposition theorem based on the VRSP that provides a Cartesian decomposition of graphs.

\end{abstract}

\begin{keyword}
Vertex-Removing Synchronised Graph Product 
\sep Product Graph
\sep Graph Decomposition
\sep Synchronising Processes


DOI: 

\end{keyword}

\end{frontmatter}

\section{Introduction}\label{sec:intro}
Recently, we have introduced~\cite{dam} and modified~\cite{ejgta1} two graph-decomposition theorems based on a new graph product, motivated by applications in the context of synchronising periodic real-time processes, in particular in the field of robotics. 
More on the background, definitions and applications can be found in two conference contributions \cite{boode2014cpa, boode2013cpa}, three journal papers~\cite{ejgta1,dam,ejgta}, on ArXiv~\cite{mod-arxiv} and the thesis of the author~\cite{boodethesis}. 
We repeat some of the background, definitions and theorems here for convenience, and present the decomposition theorem that we state and prove in Section~\ref{sec:decomp}. 

The decomposition of graphs is well known in the literature.
For example, a decomposition can be based on the partition of a graph into edge disjoint subgraphs. 
In our case, the decomposition is based on the contraction of a subset of the vertices of the graph, in such a manner that if $V'\subset V(G)$ is contracted giving $G'$ and $V''\subset V(G)$ is contracted giving $G''$ we have that the vertex-removing synchronised product (VRSP) of $G'$ and $G''$ is isomorphic to $G$.
In this contribution, we have a series of contractions of $G$ where each contraction replaces a set of vertices by one vertex (and the deletion of the arcs with both ends in the contracted set) giving a graph $G'$ and another series of contractions where each contraction replaces a set of vertices by one vertex (and the deletion of the arcs with both ends in the contracted set) giving a graph $G''$ such that $G$ is isomorphic to the VRSP of $G'$ and $G''$.

The rest of the paper is organised as follows.
In the next two sections, we first recall the formal graph definitions (in Section~\ref{sec:term}).
Then, we state and prove (in Section~\ref{sec:decomp}) the decomposition theorem.

\section{Terminology and notation}\label{sec:term}
In order to avoid duplication we refer the interested reader to \cite{dam} or \cite{mod-arxiv} for background, definitions and more details.
Furthermore, we use the textbook of Bondy and Murty~\cite{GraphTheory} for terminology and notation we have not specified here, or in \cite{dam} or in \cite{mod-arxiv}. 
For convenience, we repeat a few definitions that are especially important for this note.

Let $G$ be an edge-labelled acyclic directed multigraph with a vertex set $V$, an arc set $A$, a set of label pairs $L$ and two mappings.
The first mapping $\mu: A\rightarrow V\times V$ is an incidence function that identifies the {\em tail\/} and {\em head\/} of each arc $a\in A$. 
In particular, $\mu(a)=(u,v)$ means that the arc $a$ is directed from $u\in V$ to $v\in V$, where $tail(a)=u$ and $head(a)=v$. We also call $u$ and $v$ the {\em ends\/} of $a$. 
The second mapping $\lambda :A\rightarrow L$ assigns a label pair $\lambda(a)=(\ell(a),t(a))$ to each arc $a\in A$, where $\ell(a)$ is a string representing the (name of an) action and $t(a)$ is the {\em weight\/} of the arc $a$.

Let $X$ be a nonempty proper subset of $V(G)$. 
By {\em contracting $X$\/} we mean replacing $X$ by a new vertex $\tilde{x}$, deleting all arcs with both ends in $X$, replacing each arc $a\in A(G)$ with $\mu(a)=(u,v)$ for $u\in X$ and $v\in V(G)\backslash X$ by an arc $c$ with $\mu(c)=(\tilde{x},v)$ and $\lambda (c)=\lambda(a)$, and replacing each arc $b\in A(G)$ with $\mu(b)=(u,v)$ for $u\in V(G)\backslash X$ and $v\in X$ by an arc $d$ with $\mu(d)=(u,\tilde{x})$ and $\lambda (d)=\lambda(b)$. We denote the resulting graph as $G/X$, and say that $G/X$ is the {\em contraction of $G$ with respect to $X$\/}. 



Informally, the vertex-removing synchronised product (VRSP), denoted by $\boxbackslash$, starts from the well-known Cartesian product, denoted by $\Box$, and is based on a reduction of the number of arcs and vertices due to the presence of synchronising arcs, i.e., arcs with the same label. This reduction is done in two steps: in the first step synchronising pairs of arcs from $G_1$ and $G_2$ are replaced by one (diagonal) arc, giving the intermediate product denoted by $\boxtimes$; in the second step, vertices (and the arcs with that vertex as a tail) are removed one by one if they have $level > 0$ in the Cartesian product but $level=0$ in what is left of the intermediate product.

An arc $a\in A(G_i)$ with label pair $\lambda(a)$ is a \emph{synchronising arc} with respect to $G_j$, if and only if there exists an arc $b\in A(G_j)$ with label pair $\lambda(b)$ such that $\lambda(a)=\lambda(b)$.
Furthermore, an arc $a$ with label pair $\lambda(a)$ of $G_i\boxtimes G_j$ or $G_i\boxbackslash G_j$ is a \emph{synchronous} arc, whenever there exist a pair of arcs $a_i\in A(G_i)$ and $a_j\in A(G_j)$ with $\lambda(a)=\lambda(a_i)=\lambda(a_j)$.
Analogously, an arc $a$ with label pair $\lambda(a)$ of $G_i\boxtimes G_j$ or $G_i\boxbackslash G_j$ is an \emph{asynchronous} arc, whenever $\lambda(a)\notin L_i$ or $\lambda(a)\notin L_j$.

But, for this contribution, we consider the Cartesian part of the VRSP only.
This is due to the fact that the decomposed graphs do not contain synchronising arcs with respect to each other and therefore the VRSP is an analogue of the Cartesian product.

\section{The Cartesian graph-decomposition theorem.}\label{sec:decomp}
We assume that the graphs we want to decompose are connected; if not, we can apply our decomposition results to the components separately. 
We continue with an analogue of a result by Sabidussi, where it is stated that \enquote{every connected graph of \emph{finite type} has a unique prime factor decomposition with respect to Cartesian multiplication~\cite{Sabidussi}}.
In our case, we have a special case of Sabidussi's connected graph, because our graphs are directed, acyclic and have labelled arcs. 
For the Cartesian multiplication, it is sufficient if the graphs used for the multiplication do not share a label.

In Figure~\ref{CartDecomposition}, we give a simple example of this decomposition.
On the left of Figure~\ref{CartDecomposition}, we have the graph $G$ consisting of subgraphs $G'_1,\ldots, G'_3,G''_1,\ldots,G''_4$.
The vertex set of $G'_i$ is replaced by $\tilde{y}_i$ giving the graph  $G/V(G'_1)/\ldots/V(G'_3)$ (shown in the middle of Figure~\ref{CartDecomposition}) and the vertex set of $G''_j$ is replaced by $\tilde{x}_j$ giving the graph  $G/V(G''_1)/\ldots/V(G''_4)$ (shown at the upper left of Figure~\ref{CartDecomposition}).
The VRSP of $G/V(G'_1)/\ldots/V(G'_3)$ and $G/V(G''_1)/\ldots/V(G''_4)$ then gives the graph on the left of Figure~\ref{CartDecomposition}, $G\cong G/V(G'_1)/\ldots/V(G'_3)\boxbackslash G/V(G''_1)/\ldots/V(G''_4)$.
Because there are no synchronising arcs in $G/V(G'_1)/\ldots/V(G'_3)$ with respect to $G/V(G''_1)/\ldots/V(G''_4)$, we have that the Cartesian product, the intermediate stage and the VRSP of these graphs are the same. Therefore, $G\cong G/V(G'_1)/\ldots/V(G'_3)\Box G/V(G''_1)/\ldots/V(G''_4)= G/V(G'_1)/\ldots/V(G'_3)\boxtimes$ $G/V(G''_1)/\ldots/V(G''_4)= G/V(G'_1)/\ldots/V(G'_3)\boxbackslash G/V(G''_1)/\ldots/V(G''_4)$.
\begin{figure}[H]
\begin{center}
\resizebox{1.0\textwidth}{!}{
\begin{tikzpicture}[-{Stealth[length=5mm, width=4mm]},auto,node distance=2.5cm,
  main node/.style={circle,fill=blue!10,draw, font=\sffamily\Large\bfseries}]
  \tikzset{VertexStyle/.append style={
  font=\itshape\large, shape = circle,inner sep = 2pt, outer sep = 0pt,minimum size = 20 pt,draw}}
  \tikzset{LabelStyle/.append style={font = \itshape}}
  \SetVertexMath
  \def\x{0.0}
  \def\y{1.0}
  \def\dx{-17.0}
  \def\dy{-13.0}
\node at (\x+0.5+\dx,3+\y+3+\dy) {$G$};
\node at (\x+4.5+\dx,\y+5.4+\dy) {$G'_1$};
\node at (\x+1.5+\dx,\y+4.4+\dy) {$G''_1$};
\node at (\x+4.5+\dx,\y+2.6+\dy) {$G'_2$};
\node at (\x+5.5+\dx,\y+4.4+\dy) {$G''_2$};
\node at (\x+4.5+\dx,\y-0.4+\dy) {$G'_3$};
\node at (\x+9.5+\dx,\y+4.4+\dy) {$G''_3$};
\node at (\x+13.5+\dx,\y+4.4+\dy){$G''_4$};
\node at (\x+5.5+4,\y-1-2) {$G/V(G''_1)/V(G''_2)/V(G''_3)/V(G''_4)$};
\node at (\x+0.5+1,\y-5-1) {$G/V(G'_1)/V(G'_2)/V(G'_3)$};
\node at (\x+7.5+2.5,\y-3-2-1) {$G/V(G'_1)/V(G'_2)/V(G'_3)\boxbackslash G/V(G''_1)/V(G''_2)/V(G''_3)/V(G''_4)$};
  \def\x{0.5+\dx}
  \def\y{6.0+1+\dy}
  \Vertex[x=\x+1.5, y=\y+0.0,L={u_1}]{u_2}
  \Vertex[x=\x+1.5, y=\y-3.0,L={u_2}]{u_3}
  \Vertex[x=\x+1.5, y=\y-6.0,L={u_3}]{u_4}
  \Vertex[x=\x+5.5, y=\y+0.0,L={u_4}]{u_5}
  \Vertex[x=\x+5.5, y=\y-3,L={u_5}]{u_6}
  \Vertex[x=\x+5.5, y=\y-6,L={u_6}]{u_7}
  \Vertex[x=\x+9.5, y=\y+0.0,L={u_7}]{u_8}
  \Vertex[x=\x+9.5, y=\y-3.0,L={u_8}]{u_9}
  \Vertex[x=\x+9.5, y=\y-6.0,L={u_{9}}]{u_10}
  \Vertex[x=\x+13.5, y=\y+0,L={u_{10}}]{u_11}
  \Vertex[x=\x+13.5, y=\y-3,L={u_{11}}]{u_12}
  \Vertex[x=\x+13.5, y=\y-6,L={u_{12}}]{u_13}

  \Edge[label = a](u_2)(u_3) 
  \Edge[label = a](u_3)(u_4) 
  \Edge[label = a](u_5)(u_6) 
  \Edge[label = a](u_6)(u_7) 
  \Edge[label = a](u_8)(u_9) 
  \Edge[label = a](u_9)(u_10) 
  \Edge[label = a](u_11)(u_12) 
  \Edge[label = a](u_12)(u_13) 
  \Edge[label = b](u_2)(u_5) 
  \Edge[label = b](u_5)(u_8) 
  \Edge[label = c](u_8)(u_11) 
  \Edge[label = b](u_3)(u_6) 
  \Edge[label = b](u_6)(u_9) 
  \Edge[label = c](u_9)(u_12) 
  \Edge[label = b](u_4)(u_7) 
  \Edge[label = b](u_7)(u_10) 
  \Edge[label = c](u_10)(u_13) 
  
  \Edge(u_2)(u_3) 
  \Edge(u_3)(u_4) 
  \Edge(u_5)(u_6) 
  \Edge(u_6)(u_7) 
  \Edge(u_8)(u_9) 
  \Edge(u_9)(u_10) 
  \Edge(u_11)(u_12) 
  \Edge(u_12)(u_13) 
  \Edge(u_2)(u_5) 
  \Edge(u_5)(u_8) 
  \Edge(u_8)(u_11) 
  \Edge(u_3)(u_6) 
  \Edge(u_6)(u_9) 
  \Edge(u_9)(u_12) 
  \Edge(u_4)(u_7) 
  \Edge(u_7)(u_10) 
  \Edge(u_10)(u_13) 
  
  \def\x{4.0}
  \def\y{-3}
  \Vertex[x=\x+1, y=\y+0.0,L={\tilde{x_1}}]{s_1}
  \Vertex[x=\x+4, y=\y+0.0,L={\tilde{x_2}}]{s_2}
  \Vertex[x=\x+7, y=\y+0.0,L={\tilde{x_3}}]{s_3}
  \Vertex[x=\x+10, y=\y+0.0,L={\tilde{x_4}}]{s_4}
  
  \Edge[label = b](s_1)(s_2) 
  \Edge[label = b](s_2)(s_3) 
  \Edge[label = c](s_3)(s_4) 
  \Edge(s_1)(s_2) 
  \Edge(s_2)(s_3) 
  \Edge(s_3)(s_4) 

  \def\x{+1.5}
  \def\y{-3.0}
  \Vertex[x=\x+0, y=\y-3.0,L={\tilde{y}_1}]{t_1}
  \Vertex[x=\x+0, y=\y-6.0,L={\tilde{y}_2}]{t_2}
  \Vertex[x=\x+0, y=\y-9.0,L={\tilde{y}_3}]{t_3}

  \Edge[label = a](t_1)(t_2) 
  \Edge[label = a](t_2)(t_3) 
  \Edge(t_1)(t_2) 
  \Edge(t_2)(t_3)

\tikzset{VertexStyle/.append style={
  font=\itshape\large,shape = rounded rectangle,inner sep = 0pt, outer sep = 0pt,minimum size = 20 pt,draw}}

  \def\x{2.0}
  \def\y{-6.0}
  \Vertex[x=\x+3.0, y=\y-0.0,L={(\tilde{y}_1,\tilde{x}_1)}]{t_1s_1}
  \Vertex[x=\x+6.0, y=\y-0.0,L={(\tilde{y}_1,\tilde{x}_2)}]{t_1s_2}
  \Vertex[x=\x+9.0, y=\y-0.0,L={(\tilde{y}_1,\tilde{x}_3)}]{t_1s_3}
  \Vertex[x=\x+12.0, y=\y-0.0,L={(\tilde{y}_1,\tilde{x}_4)}]{t_1s_4}
  \def\x{2.0}
  \def\y{-9.0}
  \Vertex[x=\x+3.0, y=\y-0.0,L={(\tilde{y}_2,\tilde{x}_1)}]{t_2s_1}
  \Vertex[x=\x+6.0, y=\y-0.0,L={(\tilde{y}_2,\tilde{x}_2)}]{t_2s_2}
  \Vertex[x=\x+9.0, y=\y-0.0,L={(\tilde{y}_2,\tilde{x}_3)}]{t_2s_3}
  \Vertex[x=\x+12.0, y=\y-0.0,L={(\tilde{y}_2,\tilde{x}_4)}]{t_2s_4}
  \def\x{2.0}
  \def\y{-12.0}
  \Vertex[x=\x+3.0, y=\y-0.0,L={(\tilde{y}_3,\tilde{x}_1)}]{t_3s_1}
  \Vertex[x=\x+6.0, y=\y-0.0,L={(\tilde{y}_3,\tilde{x}_2)}]{t_3s_2}
  \Vertex[x=\x+9.0, y=\y-0.0,L={(\tilde{y}_3,\tilde{x}_3)}]{t_3s_3}
  \Vertex[x=\x+12.0, y=\y-0.0,L={(\tilde{y}_3,\tilde{x}_4)}]{t_3s_4}

  \Edge[label = b](t_1s_1)(t_1s_2) 
  \Edge[label = b](t_1s_2)(t_1s_3) 
  \Edge[label = c](t_1s_3)(t_1s_4) 
  \Edge(t_1s_1)(t_1s_2) 
  \Edge(t_1s_2)(t_1s_3) 
  \Edge(t_1s_3)(t_1s_4) 

  \Edge[label = b](t_2s_1)(t_2s_2) 
  \Edge[label = b](t_2s_2)(t_2s_3) 
  \Edge[label = c](t_2s_3)(t_2s_4) 
  \Edge(t_2s_1)(t_2s_2) 
  \Edge(t_2s_2)(t_2s_3) 
  \Edge(t_2s_3)(t_2s_4) 

  \Edge[label = b](t_3s_1)(t_3s_2) 
  \Edge[label = b](t_3s_2)(t_3s_3) 
  \Edge[label = c](t_3s_3)(t_3s_4) 
  \Edge(t_3s_1)(t_3s_2) 
  \Edge(t_3s_2)(t_3s_3) 
  \Edge(t_3s_3)(t_3s_4)

  \Edge[label = a](t_1s_1)(t_2s_1) 
  \Edge[label = a](t_1s_2)(t_2s_2) 
  \Edge[label = a](t_1s_3)(t_2s_3)
  \Edge[label = a](t_1s_4)(t_2s_4)  
  \Edge(t_1s_1)(t_2s_1) 
  \Edge(t_1s_2)(t_2s_2) 
  \Edge(t_1s_3)(t_2s_3)
  \Edge(t_1s_4)(t_2s_4)  
 
  \Edge[label = a](t_2s_1)(t_3s_1) 
  \Edge[label = a](t_2s_2)(t_3s_2) 
  \Edge[label = a](t_2s_3)(t_3s_3)
  \Edge[label = a](t_2s_4)(t_3s_4)  
  \Edge(t_2s_1)(t_3s_1) 
  \Edge(t_2s_2)(t_3s_2) 
  \Edge(t_2s_3)(t_3s_3)
  \Edge(t_2s_4)(t_3s_4)

  \def\x{1.7+\dx}
  \def\y{5.9+1+\dy}
\draw[circle, -,dashed, very thick,rounded corners=18pt] (\x-0.5,\y+0.0)--(\x-0.5,\y+0.7) --(\x+1.0,\y+0.7) -- (\x+1.0,\y-6.6) -- (\x-0.5,\y-6.6) --  (\x-0.5,\y+0.0);
  \def\x{5.7+\dx}
  \def\y{5.9+1+\dy}
\draw[circle, -,dashed, very thick,rounded corners=18pt] (\x-0.5,\y+0.0)--(\x-0.5,\y+0.7) --(\x+1.0,\y+0.7) -- (\x+1.0,\y-6.6) -- (\x-0.5,\y-6.6) --  (\x-0.5,\y+0.0);
  \def\x{9.7+\dx}
  \def\y{5.9+1+\dy}
\draw[circle, -,dashed, very thick,rounded corners=18pt] (\x-0.5,\y+0.0)--(\x-0.5,\y+0.7) --(\x+1.0,\y+0.7) -- (\x+1.0,\y-6.6) -- (\x-0.5,\y-6.6) --  (\x-0.5,\y+0.0);
  \def\x{13.7+\dx}
  \def\y{5.9+1+\dy}
\draw[circle, -,dashed, very thick,rounded corners=18pt] (\x-0.5,\y+0.0)--(\x-0.5,\y+0.7) --(\x+1.0,\y+0.7) -- (\x+1.0,\y-6.6) -- (\x-0.5,\y-6.6) --  (\x-0.5,\y+0.0);

  \def\x{1.7+\dx}
  \def\y{5.1+1+\dy}
\draw[circle, -,dashed, very thick,rounded corners=18pt] (\x+2.8,\y+0.0)--(\x-0.5,\y+0.0)--(\x-0.5,\y+1.5)--(\x+13.0,\y+1.5) --(\x+13.0,\y) -- (\x+2.8,\y+0.0);

  \def\x{1.7+\dx}
  \def\y{2.3+1+\dy}
\draw[circle, -,dashed, very thick,rounded corners=18pt] (\x+2.8,\y+0.0)--(\x-0.5,\y+0.0)--(\x-0.5,\y+1.5)--(\x+13.0,\y+1.5) --(\x+13.0,\y) -- (\x+2.8,\y+0.0);

  \def\x{1.7+\dx}
  \def\y{-0.7+1+\dy}
\draw[circle, -,dashed, very thick,rounded corners=18pt] (\x+2.8,\y+0.0)--(\x-0.5,\y+0.0)--(\x-0.5,\y+1.5)--(\x+13.0,\y+1.5) --(\x+13.0,\y) -- (\x+2.8,\y+0.0);

\end{tikzpicture}
}
\end{center}
\caption{Decomposition of $G$ giving $G\cong G/V(G'_1)/V(G'_2)/V(G'_3)\boxbackslash G/V(G''_1)/V(G''_2)/V(G''_3)/V(G''_4)$.}
  \label{CartDecomposition}
\end{figure}
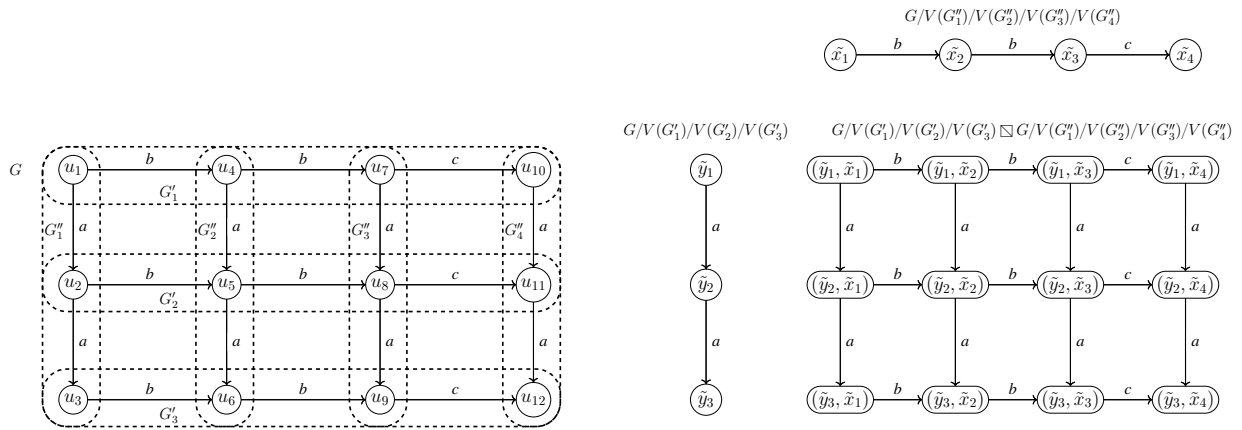

In general, we have the (repeated) contraction of a vertex set of a graph $G$ into the graphs $G'_i$, $G''_j$ such that $G'_i\boxbackslash G''_j\cong G$.
Then, the requirements for the (repeated) contraction of a vertex set of the graph $G$ are as follows: $G=\overundercup{i=1}{m}G'_i\cup \overundercup{j=1}{n}G''_j$, $G'_i\cong G'_j$, $G''_i\cong G''_j$, $G/V(G'_1)/\ldots/V(G'_m)\cong G''_i$, $G/V(G''_1)/\ldots/V(G''_n)\cong G'_i$, $V(G)=\overundercup{i=1}{m}V(G'_i)$, $V(G)=\overundercup{j=1}{n}V(G''_j)$,  $A(G)=\overundercup{j=1}{m}A(G'_i)\cup \overundercup{j=1}{n}A(G''_j)$, $L(G)=L(G'_i)\cup L(G''_j)$ and $L(G'_i)\cap L(G''_j)=\emptyset$.

This leads to theorem~\ref{theorem1}.
Note that the line of proof is similar to the line of proof of our contribution in~\cite{dam}.
\begin{theorem}\label{theorem1}
Let $G$ be a graph. If $G=\overundercup{i=1}{m}G'_i\cup \overundercup{j=1}{n}G''_j$, $G'_i\cong G'_j$, $G''_i\cong G''_j$, $G/V(G'_1)/\ldots/V(G'_m)$ $\cong G''_i$, $G/V(G''_1)/\ldots/V(G''_n)\cong G'_i$, $V(G)=\overundercup{i=1}{m}V(G'_i)=\overundercup{j=1}{n}V(G''_j)$, and $L(G'_i)\cap L(G''_j)=\emptyset$.
Then $G\cong G/V(G'_1)/\ldots/V(G'_m)\boxbackslash G/V(G''_1)/\ldots/V(G''_n)$.
\end{theorem}
\begin{proof}
Let $G_1=G/V(G'_1)/\ldots/V(G'_m)$ and $G_2=G/V(G''_1)/\ldots$ $/V(G''_n)$.
It suffices to define a mapping $\phi: V(G)\rightarrow V(G_1\boxbackslash G_2)$ and to prove that $\phi$ is an isomorphism from $G$ to $G_1\boxbackslash G_2$.
Let $\tilde{x}_i$ and $\tilde{y}_j$ be the new vertices replacing the sets of vertices $V(G''_i)$ and $V(G'_j)$ when defining $G_2$ and $G_1$, respectively.
Consider the mapping $\phi:V(G)\rightarrow V(G_1\boxbackslash G_2)$ defined by $\phi(u)=(\tilde{y}_j,\tilde{x}_i)$ where $u\in G''_i\cap G'_j$.
\begin{claim}\label{claim-inj}
$\phi$ is an injective map from $V(G)$ to $V(G_1\boxbackslash G_2)$.
\end{claim}
\begin{proof}
Because $V(G)=\overundercup{i=1}{m}V(G'_i)=\overundercup{j=1}{n}V(G''_j)$, we have that for each vertex $u\in V(G)$ there exist subgraphs $G'_i$ and $G''_j$ such that $u\in G'_i\cap G''_j$.
Such a vertex $u$ is uniquely defined. 
If there exist vertices $u,v\in G'_i\cap G''_j$ then because $G_1\cong G''_i$ and $G_2\cong G'_i$ we have $u=v$.
By contradiction, suppose $u\neq v$ then both $u$ and $v$ are replaced by $\tilde{x}_i$ in $G_2$ and by $\tilde{y}_j$ in $G_1$. But then, because $u,v\in G'_i$ it follows that $G_2\ncong G'_i$ and because $u,v\in G''_i$ it follows that $G_1\ncong G''_i$.
Now, the contraction of $G$ giving $G_1$ is an injective map from all $u\in G''_i$ to $\tilde{x}_i$, the contraction of $G$ giving $G_2$ is an injective map from all $u\in G'_j$ to $\tilde{y}_j$.
Then, by the definition of the Cartesian product, we have that $u$ is mapped on $(\tilde{y}_j,\tilde{x}_i)$ in $G_1\Box G_2$.
Because, all arcs $a$ of $G_1$ are not synchronizing arcs with respect to all arcs $b$ of $G_2$, there are no vertices removed from $G_1\boxtimes G_2$ and, therefore, $V(G_1\boxtimes G_2)=V(G_1\boxbackslash G_2)$.
Hence, the map $\phi$ is a injective map from $V(G)$ to $V(G_1\boxbackslash G_2)$.
\end{proof}
\begin{claim}\label{claim-sur}
$\phi$ is a surjective map from $V(G)$ to $V(G_1\boxbackslash G_2)$.
\end{claim}
\begin{proof}
Furthermore, because $\overundercup{i=1}{m} V(G'_i)=V(G)$, we have $V(G_1)=\{\tilde{y}_1,\ldots,\tilde{y}_m\}$ and, similarly, because $\overundercup{i=1}{n} V(G''_i)=V(G)$, we have $V(G_2)=\{\tilde{x}_1,\ldots,\tilde{x}_n\}$, giving $|V(G)|=|V(G_1\Box G_2)|=m\cdot n$. 
Because all arcs $a$ of $\overundercup{i=1}{m}G'_i$ are not synchronizing with respect to all arcs $b$ of $\overundercup{j=1}{n}G''_j$, due to the definition of the VRSP we have $|V(G)|=|V(G_1\Box G_2)|=|V(G_1\boxbackslash G_2)|=m\cdot n$.
Hence, we have $V(G_1\boxbackslash G_2)=\{(\tilde{y}_j,\tilde{x}_i)\mid i\in\{1,\ldots, m\},j\in \{1,\ldots, n\}$.
Therefore, there is a vertex $u\in V(G)$ for every vertex $(\tilde{y}_j,\tilde{x}_i) \in V(G_1\boxbackslash G_2)$ and the map $\phi$ is a surjective map from $V(G_1\boxbackslash G_2)$ to $V(G)$.
\end{proof}
Due to Claim~\ref{claim-inj} and Claim~\ref{claim-sur}, we have that $\phi$ is a bijection from $V(G)$ to $V(G_1\boxbackslash G_2)$.
Because the vertices $(\tilde{y}_j,\tilde{x}_i), i=1,\ldots, m, j=1,\ldots,n,$ are the only vertices in $G/G_1\boxbackslash G/G_2$, it remains to show that $\phi$ preserves the arcs and their labels. 
By definition of the Cartesian product we have that for each arc $a \in A(G_1)$ with $\mu(a)=(\tilde{y}_{j_{_1}},\tilde{y}_{j_{_2}})$, there is a set of arcs $\{a'\mid \mu(a')=((\tilde{y}_{j_{_1}},\tilde{x}_{i}),(\tilde{y}_{j_{_2}},\tilde{x}_{i})),\lambda(a)=\lambda(a'),\tilde{x}_i\in G_2\}\subseteq A(G_1\Box G_2)$.
Because this is true for all arcs $a$ of $G_1$ and all vertices $\tilde{x}_{i}$ of $G_2$, we have that each graph $G'_i$ induced by $(\tilde{y}_{j},\tilde{x}_{i}),\tilde{y}_{j}\in G_1$ is isomorphic to $G_1$, analogously, for all arcs $a$ of $G_2$ and all vertices $\tilde{y}_{j}$ of $G_1$, we have that each graph $G''_j$ induced by $(\tilde{y}_{j},\tilde{x}_{i}),\tilde{x}_{i}\in G_2$ is isomorphic to $G_2$.
Together with $G_1\cong G'_i$ and $G_2\cong G''_j$, we have that the map $\phi$ preserves the arcs and their labels and therefore $\phi$ is an isomorphism from $G$ to $G_1\boxbackslash G_2$.
Hence, $G\cong G_1\boxbackslash G_2$.
This completes the proof of Theorem~\ref{theorem1}.
\end{proof}

In the following two figures, we will show that all requirements of Theorem~\ref{theorem1} are necessary to come to the conclusion that $\phi$ is an isomorphism from $G$ to $G_1\boxbackslash G_2$ and therefore $G\cong G_1\boxbackslash G_2$.
In Figure~\ref{CartDecomposition1}, we show that $G=\overundercup{i=1}{m}G'_i\cup \overundercup{j=1}{n}G''_j$ and $V(G)=\overundercup{i=1}{m}V(G'_i)=\overundercup{j=1}{n}V(G''_j)$ must not be violated.
In Figure~\ref{CartDecomposition2}, we show that $G'_i\cong G'_j$, $G''_i\cong G''_j$, $G/V(G'_1)/\ldots/V(G'_m)\cong G''_i$ and $G/V(G''_1)/\ldots/V(G''_n)\cong G'_i$, must not be violated.
For $L(G'_i)\cap L(G''_j)=\emptyset$, it is obvious that the decomposition will fail if there are synchronizing arcs in $G_1$ with respect to $G_2$. 
Therefore, we do not show this in an example figure.

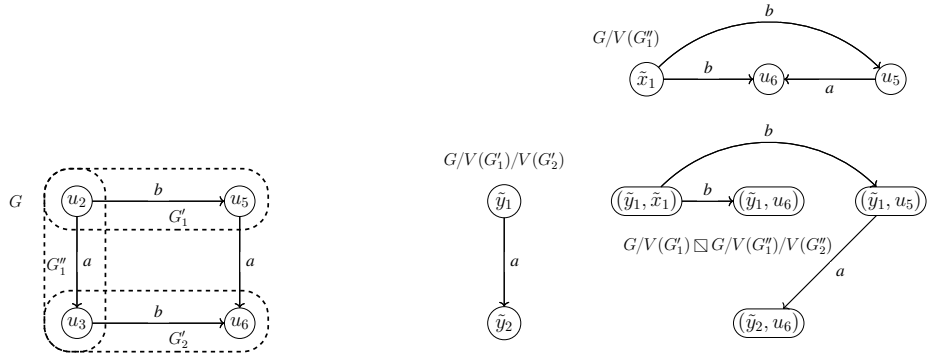
\begin{figure}[H]
\begin{center}
\resizebox{0.75\textwidth}{!}{
\begin{tikzpicture}[-{Stealth[length=5mm, width=4mm]},midway ,auto,node distance=2.5cm,
  main node/.style={circle,fill=blue!10,draw, font=\sffamily\Large\bfseries}]
  \tikzset{VertexStyle/.append style={
  font=\itshape\large, shape = circle,inner sep = 2pt, outer sep = 0pt,minimum size = 20 pt,draw}}
  \tikzset{LabelStyle/.append style={font = \itshape}}
  \SetVertexMath
  \def\x{0.0}
  \def\y{1.0}
  \def\dx{-11.0}
  \def\dy{-10.0}
\node at (\x+0.5+\dx,3+\y+0+\dy) {$G$};
\node at (\x+1.5+\dx,\y+1.4+\dy) {$G''_1$};
\node at (\x+4.5+\dx,\y+2.6+\dy) {$G'_1$};
\node at (\x+4.5+\dx,\y-0.4+\dy) {$G'_2$};
\node at (\x+4.5,\y-1-2) {$G/V(G''_1)$};
\node at (\x+0.5+1,\y-5-1) {$G/V(G'_1)/V(G'_2)$};
\node at (\x+8.5+1.5-6+3,\y-5.1-2-1) {$G/V(G'_1)\boxbackslash G/V(G''_1)/V(G''_2)$};
  \def\x{0.5}
  \def\y{6.0+1}

  \Vertex[x=\x+1.5+\dx, y=\y-3.0+\dy,L={u_2}]{u_3}
  \Vertex[x=\x+1.5+\dx, y=\y-6.0+\dy,L={u_3}]{u_4}
  \Vertex[x=\x+5.5+\dx, y=\y-3+\dy,L={u_5}]{u_6}
  \Vertex[x=\x+5.5+\dx, y=\y-6+\dy,L={u_6}]{u_7}

  \Edge[label = a](u_3)(u_4) 
  \Edge[label = a](u_6)(u_7) 
  \Edge[label = b](u_3)(u_6) 
  \Edge[label = b](u_4)(u_7) 

  \Edge(u_3)(u_4) 
  \Edge(u_6)(u_7) 
  \Edge(u_3)(u_6) 
  \Edge(u_4)(u_7) 

  \def\x{4.0}
  \def\y{-3}
  \Vertex[x=\x+1, y=\y+0.0,L={\tilde{x}_1}]{s_1}
  \Vertex[x=\x+7, y=\y+0.0,L={u_5}]{s_3}
  \Vertex[x=\x+4, y=\y+0.0,L={u_6}]{s_2}
  
  \Edge[label = b](s_1)(s_2) 
  \Edge[label = b,style={bend left=45,min distance=1cm}](s_1)(s_3) 
  \Edge[label = a](s_3)(s_2) 

  \Edge(s_1)(s_2) 
    \Edge[style={bend left=45,min distance=1cm}](s_1)(s_3) 
  \Edge(s_3)(s_2) 

  \def\x{+1.5}
  \def\y{-3.0}
  \Vertex[x=\x+0, y=\y-3.0,L={\tilde{y}_1}]{t_1}
  \Vertex[x=\x+0, y=\y-6.0,L={\tilde{y}_2}]{t_2}

  \Edge[label = a](t_1)(t_2) 
  \Edge(t_1)(t_2)

\tikzset{VertexStyle/.append style={
  font=\itshape\large,shape = rounded rectangle,inner sep = 0pt, outer sep = 0pt,minimum size = 20 pt,draw}}

  \def\x{2.0}
  \def\y{-7.0}
  \def\x{2.0}
  \def\y{-6.0}
  \Vertex[x=\x+3.0, y=\y-0.0,L={(\tilde{y}_1,\tilde{x}_1)}]{t_1s_1}
  \Vertex[x=\x+6.0, y=\y-0.0,L={(\tilde{y}_1,u_6)}]{t_1s_2}
  \Vertex[x=\x+9.0, y=\y-0.0,L={(\tilde{y}_1,u_5)}]{t_1s_3}

  \def\x{2.0}
  \def\y{-9.0}
  \Vertex[x=\x+6.0, y=\y-0.0,L={(\tilde{y}_2,u_6)}]{t_2s_2}

    \def\x{2.0}
  \def\y{-19.0}

  \Edge[label = b](t_1s_1)(t_1s_2) 
  \Edge[label = a](t_1s_3)(t_2s_2) 
  \Edge[label = b,style={bend left=45,min distance=1cm}](t_1s_1)(t_1s_3) 
  
  \Edge(t_1s_1)(t_1s_2) 
    \Edge[style={bend left=45,min distance=1cm}](t_1s_1)(t_1s_3)

  \def\x{1.7+\dx}
  \def\y{5.9-1.8+\dy}
\draw[circle, -,dashed, very thick,rounded corners=18pt] (\x-0.5,\y+0.0)--(\x-0.5,\y+0.7) --(\x+1.0,\y+0.7) -- (\x+1.0,\y-3.8) -- (\x-0.5,\y-3.8) --  (\x-0.5,\y+0.0);

  \def\x{1.7+\dx}
  \def\y{2.3+1+\dy}
\draw[circle, -,dashed, very thick,rounded corners=18pt] (\x+2.8,\y+0.0)--(\x-0.5,\y+0.0)--(\x-0.5,\y+1.5)--(\x+5.0,\y+1.5) --(\x+5.0,\y) -- (\x+2.8,\y+0.0);

  \def\x{1.7+\dx}
  \def\y{-0.7+1+\dy}
\draw[circle, -,dashed, very thick,rounded corners=18pt] (\x+2.8,\y+0.0)--(\x-0.5,\y+0.0)--(\x-0.5,\y+1.5)--(\x+5.0,\y+1.5) --(\x+5.0,\y) -- (\x+2.8,\y+0.0);

\end{tikzpicture}
}
\end{center}
\caption{Decomposition of $G$ for which $G\neq \bigcup\limits_{i=1}^{2}G'_i\cup \bigcup\limits_{j=1}^{1}G''_j$ and $V(G)=\bigcup\limits_{i=1}^{2}V(G'_i)\neq\bigcup\limits_{j=1}^{1}V(G''_j)$ and, therefore, $G\ncong G/G'_1/G'_2\boxbackslash G/G''_1$.}
  \label{CartDecomposition1}
\end{figure}

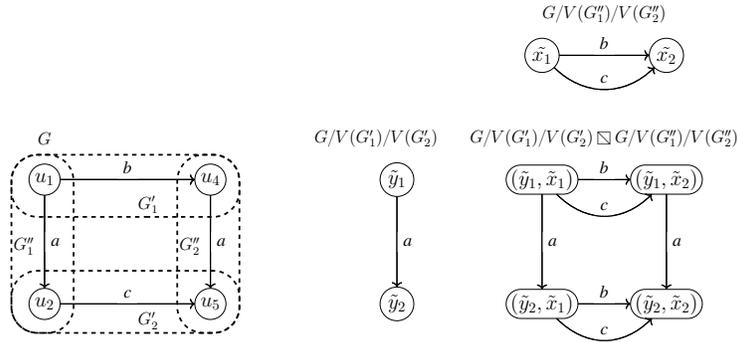
\begin{figure}[H]
\begin{center}
\resizebox{0.6\textwidth}{!}{
\begin{tikzpicture}[-{Stealth[length=5mm, width=4mm]},auto,node distance=2.5cm,decoration={
    markings,
    mark=at position 0.5 with {\arrow{>}}},
  main node/.style={circle,fill=blue!10,draw, font=\sffamily\Large\bfseries}]
  \tikzset{VertexStyle/.append style={
  font=\itshape\large, shape = circle,inner sep = 2pt, outer sep = 0pt,minimum size = 20 pt,draw}}
  \tikzset{LabelStyle/.append style={font = \itshape}}
  \SetVertexMath
  \def\x{0.0}
  \def\y{1.0}
  \def\dx{-9.0}
  \def\dy{-9.0}
\node at (\x+2.0+\dx,3+\y+\dy) {$G$};
\node at (\x+4.5+\dx,\y+5.4-4+\dy) {$G'_1$};
\node at (\x+1.5+\dx,\y+4.4-4+\dy) {$G''_1$};
\node at (\x+4.5+\dx,\y+2.6-4+\dy) {$G'_2$};
\node at (\x+5.5+\dx,\y+4.4-4+\dy) {$G''_2$};
\node at (\x+6.5,\y-1-2) {$G/V(G''_1)/V(G''_2)$};
\node at (\x+1,\y-5-1) {$G/V(G'_1)/V(G'_2)$};
\node at (\x+3.5+3,\y-3-2-1) {$G/V(G'_1)/V(G'_2)\boxbackslash G/V(G''_1)/V(G''_2)$};
  \def\x{0.5}
  \def\y{6.0+1-4}

  \Vertex[x=\x+1.5+\dx, y=\y+0.0+\dy,L={u_1}]{u_2}
  \Vertex[x=\x+1.5+\dx, y=\y-3.0+\dy,L={u_2}]{u_3}
  \Vertex[x=\x+5.5+\dx, y=\y+0.0+\dy,L={u_4}]{u_5}
  \Vertex[x=\x+5.5+\dx, y=\y-3+\dy,L={u_5}]{u_6}

  \Edge[label = a](u_2)(u_3) 
  \Edge[label = a](u_5)(u_6) 
  \Edge[label = b](u_2)(u_5) 
  \Edge(u_2)(u_5) 
  \Edge[label = c](u_3)(u_6) 

  \Edge(u_2)(u_3) 
  \Edge(u_5)(u_6) 
  \Edge(u_2)(u_5) 
  \Edge(u_3)(u_6) 

  \def\x{4.0}
  \def\y{-3}
  \Vertex[x=\x+1, y=\y+0.0,L={\tilde{x_1}}]{s_1}
  \Vertex[x=\x+4, y=\y+0.0,L={\tilde{x_2}}]{s_2}
  
  \Edge[label = b](s_1)(s_2) 
  \Edge(s_1)(s_2) 
  \Edge[label = c,style={bend left=-45,min distance=1cm}](s_1)(s_2) 
    \Edge[style={bend left=-45,min distance=1cm}](s_1)(s_2)

  \def\x{+1.5}
  \def\y{-3.0}
  \Vertex[x=\x+0, y=\y-3.0,L={\tilde{y}_1}]{t_1}
  \Vertex[x=\x+0, y=\y-6.0,L={\tilde{y}_2}]{t_2}

  \Edge[label = a](t_1)(t_2) 
  \Edge(t_1)(t_2)

\tikzset{VertexStyle/.append style={
  font=\itshape\large,shape = rounded rectangle,inner sep = 0pt, outer sep = 0pt,minimum size = 20 pt,draw}}

  \def\x{2.0}
  \def\y{-6.0}
  \Vertex[x=\x+3.0, y=\y-0.0,L={(\tilde{y}_1,\tilde{x}_1)}]{t_1s_1}
  \Vertex[x=\x+6.0, y=\y-0.0,L={(\tilde{y}_1,\tilde{x}_2)}]{t_1s_2}

  \def\x{2.0}
  \def\y{-9.0}
  \Vertex[x=\x+3.0, y=\y-0.0,L={(\tilde{y}_2,\tilde{x}_1)}]{t_2s_1}
  \Vertex[x=\x+6.0, y=\y-0.0,L={(\tilde{y}_2,\tilde{x}_2)}]{t_2s_2}

    \def\x{2.0}
  \def\y{-19.0}

  \Edge[label = b](t_1s_1)(t_1s_2) 
  \Edge[label = c,style={bend left=-45,min distance=1cm}](t_1s_1)(t_1s_2)

  \Edge(t_1s_1)(t_1s_2) 
  \Edge[style={bend left=-45,min distance=1cm}](t_1s_1)(t_1s_2) 

  \Edge[label = b](t_2s_1)(t_2s_2) 
  \Edge[label = c,style={bend left=-45,min distance=1cm}](t_2s_1)(t_2s_2) 
  \Edge(t_2s_1)(t_2s_2) 
  \Edge[style={bend left=-45,min distance=1cm}](t_2s_1)(t_2s_2)

  \Edge[label = a](t_1s_1)(t_2s_1) 
  \Edge[label = a](t_1s_2)(t_2s_2) 
  \Edge(t_1s_1)(t_2s_1) 
  \Edge(t_1s_2)(t_2s_2) 
 
  \def\x{1.7+\dx}
  \def\y{5.9+1-4+\dy}
\draw[circle, -,dashed, very thick,rounded corners=18pt] (\x-0.5,\y+0.0)--(\x-0.5,\y+0.7) --(\x+1.0,\y+0.7) -- (\x+1.0,\y-3.6) -- (\x-0.5,\y-3.6) --  (\x-0.5,\y+0.0);
  \def\x{5.7+\dx}
  \def\y{5.9+1-4+\dy}
\draw[circle, -,dashed, very thick,rounded corners=18pt] (\x-0.5,\y+0.0)--(\x-0.5,\y+0.7) --(\x+1.0,\y+0.7) -- (\x+1.0,\y-3.6) -- (\x-0.5,\y-3.6) --  (\x-0.5,\y+0.0);
  \def\x{9.7}
  \def\y{5.9+1}
  
  \def\x{1.7+\dx}
  \def\y{5.1+1-4+\dy}
\draw[circle, -,dashed, very thick,rounded corners=18pt] (\x+2.8,\y+0.0)--(\x-0.5,\y+0.0)--(\x-0.5,\y+1.5)--(\x+5.0,\y+1.5) --(\x+5.0,\y) -- (\x+2.8,\y+0.0);

  \def\x{1.7+\dx}
  \def\y{2.3+1-4+\dy}
\draw[circle, -,dashed, very thick,rounded corners=18pt] (\x+2.8,\y+0.0)--(\x-0.5,\y+0.0)--(\x-0.5,\y+1.5)--(\x+5.0,\y+1.5) --(\x+5.0,\y) -- (\x+2.8,\y+0.0);

\end{tikzpicture}
}
\end{center}
\caption{Decomposition of $G$ for which $G'_i\ncong G'_j$ and $G/V(G''_1)/\ldots/V(G''_n)\cong G'_i$.}
  \label{CartDecomposition2}
\end{figure}


\end{document}